\documentclass[12pt]{article}  

\usepackage{amsmath,amsthm,amssymb}
\usepackage{shuffle}
\usepackage[all]{xy}
\usepackage{setspace}
\allowdisplaybreaks
\usepackage{color}

\newtheorem{theorem}{Theorem}

\newtheorem{corollary}[theorem]{Corollary}

\theoremstyle{remark}

\newcommand{\Z}{{\mathbb Z}}

\newcommand{\C}{{\mathbb C}}
\newcommand{\Q}{{\mathbb Q}}

\newcommand{\wzeta}{\widetilde{\zeta}}

\title{Hecke eigenform and double Eisenstein series}
\author{Koji Tasaka\footnote{email : tasaka@ist.aichi-pu.ac.jp, School of Information Science and Technology, Aichi Prefectural University}}
\date{}

\begin{document}
\maketitle

\begin{abstract}
Cusp forms for the full modular group can be written as linear combinations of double Eisenstein series introduced by Gangl, Kaneko and Zagier. 
We give an explicit formula for decomposing a Hecke eigenform into double Eisenstein series.
\end{abstract}


\section{Main result and its proof}

Double Eisenstein series $G_{r,s}$ was first studied by Gangl, Kaneko and Zagier \cite[Section 7]{GKZ} in connection with double zeta values
\[ \zeta(r,s)=\sum_{0<n<m}\frac{1}{n^rm^s}\quad (r\ge1,s\ge2).\]
For $k$ even, they proved that the $\C$-vector space $\mathcal{DE}_k$ spanned by the Eisenstein series $G_k$ and $G_{r,s} \ (r+s=k,\ r,s\ge1)$ contains the $\C$-vector space $S_k$ of cusp forms of weight $k$ for ${\rm SL}_2(\Z)$. 
Therefore for $f\in S_k$ there are numbers $b_k,a_{r,k-r} \in \C$ such that $f=b_k G_k + \sum a_{r,k-r} G_{r,k-r}$, and its constant term as $q$-series leads to the relation $b_k \zeta(k) + \sum a_{r,k-r} \zeta(r,k-r)=0$.
This explains a phenomenon discovered by Zagier \cite[\S 8]{Z} that cusp forms give rise to linear relations among double zeta values.
Since double zeta values satisfy numerous linear relations, the expression of linear relations obtained from a cusp form is not unique.

The motivation of this paper is the question: which relations of double zeta values can give back cusp forms?
Namely, find a linear relation $b_k \zeta(k) + \sum a_{r,k-r} \zeta(r,k-r)=0$ such that $b_k G_k + \sum a_{r,k-r} G_{r,k-r}$ is a cusp form.
In this paper, we give a choice of such coefficients $b_k,a_{r,k-r} \ (1\le r\le k-1)$ by showing a canonical expression of normalized Hecke eigenforms of $S_k$ in terms of double Eisenstein series.

\

We set up notation.
For a positive integer $k\ge1$ we define the (formal) Eisenstein series as a $q$-series
\[ G_k(q)=\wzeta(k) + \frac{(-1)^{k}}{(k-1)!} \sum_{n>0} \sigma_{k-1}(n) q^n ,\]
where we let $\wzeta(k)= \zeta(k)/(2\pi \sqrt{-1})^k $ for $k\ge2$ and $\wzeta(1):=0$, and $\sigma_{k-1}(n)=\sum_{d\mid n} d^{k-1}$ is the usual divisor function.
The $q$-series $G_k(q)$ is a modular form of weight $k$ for ${\rm SL}_2(\Z)$, if $k\ge4$ is even and $q=e^{2\pi \sqrt{-1}\tau}$ with $\tau$ being a variable in the complex upper half-plane (we refer to \cite[Chapter VII]{Serre} for the basics of modular forms).
The $q$-series $G_k(q)$ with $k$ odd may not play a role in the theory of classical modular forms, but is important in connection with zeta values.

Following \cite[Section 7]{GKZ}, we define the double Eisenstein series as a $q$-series.
Let $\wzeta(r,s)= \zeta(r,s)/(2\pi \sqrt{-1})^{r+s} $ for $r\ge1,s\ge2$ and $\wzeta(r,1) = -\wzeta(1,r)-\wzeta(r+1)$ for $r\ge2$.
We also let $\wzeta(1,1)=-\frac12 \wzeta(2)$ for completeness of the definition.
Put
\[ C_{r,s}^p=\delta_{r,p}+(-1)^r \binom{p-1}{r-1}+(-1)^{p-s}\binom{p-1}{s-1}\in\Z,\]
where $\delta_{r,p}$ is Kronecker's delta.
The double Eisenstein series $G_{r,s}(q)$ is then defined for integers $r,s\ge1$ by
\[ G_{r,s}(q) = \wzeta(r,s) + \sum_{\substack{h+p=r+s\\h,p\ge1}} C_{r,s}^p g_h(q) \wzeta(p) + g_{r,s}(q) + \frac12 \varepsilon_{r,s}(q) \in \C[[q]], \]
where $q$-series $g_h(q),g_{r,s}(q),\varepsilon_{s,r}(q)$ are defined in \cite[Eqs.~(51), (53), (56)]{GKZ}: for integers $h,r,s\ge1$ we let
\begin{align*}
g_{h}(q)&= \frac{(-1)^{h}}{(h-1)!} \sum_{\substack{0<d\\ 0<c}} c^{h-1} q^{cd}, \quad g_{r,s}(q)= \frac{(-1)^{r+s}}{(r-1)!(s-1)!} \sum_{\substack{0<d_1<d_2\\ 0<c_1,c_2}} c_1^{r-1}c_2^{s-1} q^{c_1d_1+c_2d_2},\\
\varepsilon_{s,r}(q)&=\delta_{s,2}g_r^\ast(q)-\delta_{s,1}g_{r-1}^\ast (q) + \delta_{r,1} \big( g_{s-1}^\ast (q) + g_r(q) \big) + \delta_{r,1}\delta_{s,1} g_2(q)
\end{align*}
with $g_k^\ast (	q) = \frac{(-1)^k}{k!} \sum_{d,c>0} d c^k q^{dc} \ (k\ge0)$.
Note that in \cite[Section 7]{GKZ} they use the opposite convention, i.e. their double zeta value $\zeta(r,s)$ converges if $r\ge2$ and $s\ge1$.

From the period theory of modular forms, we define the completed $L$-function of $f\in S_k$ by
\[L_f^\ast(s) = \int_0^{\infty} f(it)t^{s-1}dt.\]
We set the numbers $q_{i,j}(f)\in \C 
$ for $f\in S_k$ and integers $i,j\ge1$ with $i+j=k$ by
\begin{equation}\label{eq:q} 
q_{i,j}(f) = \sum_{\substack{r+s=k\\r,s:{\rm odd}}} (-1)^{\frac{s-1}{2}} L_f^\ast(s) \binom{i-1}{s-1}. 
\end{equation}
In \cite{Popa}, $q_{i,j}(f)$ is denoted by $s_{i-1}^+(f)$ (the subscript $j$ does not appear in the right hand-side of \eqref{eq:q}, but we keep it for convenience).
Note that by the functional equation $L_f^\ast(s)=(-1)^\frac{k}{2}L_f^\ast(k-s)$ we have $q_{k-1,1}(f)=0$ for $f\in S_k$.

\

We can now formulate the main result of this paper.

\begin{theorem}\label{main}
For a normalized Hecke eigenform $f\in S_k$, we have
\begin{equation}\label{eq:main} 
\sum_{\substack{r+s=k\\r,s\ge1:{\rm odd}}} q_{r,s}(f)G_{r,s}^\frac12(q)  = \frac{(-1)^{\frac{k}{2}} L_f^\ast (k-1)}{4(k-2)!} f ,
\end{equation}
where we write $G_{r,s}^\frac12(q)=G_{r,s}(q)+\frac12 G_{r+s}(q)$.
\end{theorem}

\begin{proof}
The proof is elegantly done by combining works of Gangl, Kaneko and Zagier \cite{GKZ}, Kohnen and Zagier \cite{KZ}, and Popa \cite{Popa}.
First of all, we recall some results of \cite{GKZ}.
Let 
\[ P_{r,s}(q)=G_{r}(q)G_s(q) + (\delta_{r,2}+\delta_{s,2})\frac{G_{r+s-2}'(q)}{2(r+s-2)} ,\]
where $G'_k(q)=q\frac{d}{dq}G_k(q)$.
It was shown in \cite[Theorem 7]{GKZ} that the double Eisenstein series satisfies the double shuffle relation: for $r,s\ge1 \ (r+s\ge3)$ we have
\begin{equation}\label{eq:ds}
\begin{aligned}
P_{r,s}(q) &= G_{r,s}(q)+G_{s,r}(q)+G_{r+s}(q)\\
&=\sum_{\substack{i+j=r+s\\i,j\ge1}}\left(\binom{j-1}{r-1}+\binom{j-1}{s-1}\right) G_{i,j}(q).
\end{aligned}
\end{equation}
(Theorem 7 of \cite{GKZ} only claims that the real part without constant terms of the relation \eqref{eq:ds} holds in $q\Q[[q]]$, but they also checked validities of the imaginary part and constant terms.)
The space $\mathcal{DE}_k$ spanned by $G_{r,s} \ (r+s=k,\ r,s\ge1)$ and $G_k$ is therefore a $q$-series realization of the formal double zeta space $\mathcal{D}_k\otimes \C$ (see \cite[Eq.~(21)]{GKZ} for the definition of $\mathcal{D}_k$).
Applying Theorem 3 of \cite{GKZ} to the double Eisenstein series, we have, for a modular form $f\in M_k=S_k\oplus \C G_k(q)$, 
\begin{equation*}
\sum_{\substack{r+s=k\\r,s:{\rm even}}} q_{r,s}(f) G_{r,s} (q)= 3 \sum_{\substack{r+s=k\\r,s:{\rm odd}}} q_{r,s}(f) G_{r,s} (q) +\sum_{r+s=k}(-1)^{r-1} q_{r,s}(f)G_k(q),
\end{equation*}
where the coefficient of $G_k(q)$ in the above equation can be found in the proof of \cite[Theorem 3]{GKZ}.
From this and the first equality in \eqref{eq:ds} we have
\begin{equation}\label{eq:gkz} 
\sum_{\substack{r+s=k\\r,s:{\rm even}}} q_{r,s}(f)\big(P_{r,s} (q)-G_k(q)\big)= 6 \sum_{\substack{r+s=k\\r,s:{\rm odd}}} q_{r,s}(f) G_{r,s} (q) +2G_k(q)\sum_{r+s=k}(-1)^{r-1} q_{r,s}(f).
\end{equation}
Secondly, we use a result of \cite[Theorem 9]{KZ}: it is an extra relation among values $L_f^\ast(s)$ at $s\in\{1,3,\ldots,k-1\}$.
With $L_f^\ast(s)=(-1)^\frac{k}{2}L_f^\ast(k-s)$ and easily checked identities of binomial coefficients, Theorem 9 (ii) of \cite{KZ} can be reduced to a simple identity
\begin{equation}\label{eq:KZ} 
\sum_{\substack{r+s=k\\r,s:{\rm even}}} q_{r,s}(f) \left( \frac{\beta_r\beta_s}{\beta_k}+1\right) +  \sum_{\substack{r+s=k\\r,s:{\rm odd}}} q_{r,s}(f) =0 \quad (\forall f\in S_k),
\end{equation}
where $\beta_k=\wzeta(k)\in \Q$ for $k$ even.
Using \eqref{eq:gkz}, for a cusp form $f\in S_k$ one computes 
\begin{align*}
&\sum_{\substack{r+s=k\\r,s:{\rm even}}} q_{r,s}(f) \left(P_{r,s}(q) - \frac{\beta_r\beta_s}{\beta_k} G_k(q) \right)\\
&= \sum_{\substack{r+s=k\\r,s:{\rm even}}} q_{r,s}(f) \left(P_{r,s}(q) - G_k(q) \right) +  G_k(q)\sum_{\substack{r+s=k\\r,s:{\rm even}}} q_{r,s}(f) \left(1 - \frac{\beta_r\beta_s}{\beta_k}  \right)  \\
& = 6 \sum_{\substack{r+s=k\\r,s:{\rm odd}}} q_{r,s}(f) G_{r,s}(q) + 2 G_k(q)\sum_{r+s=k}(-1)^{r-1} q_{r,s}(f)+  G_k(q)\sum_{\substack{r+s=k\\r,s:{\rm even}}} q_{r,s}(f) \left(1 - \frac{\beta_r\beta_s}{\beta_k}  \right) \\
& = 6 \sum_{\substack{r+s=k\\r,s:{\rm odd}}} q_{r,s}(f) G_{r,s}(q) + 2 G_k(q)\sum_{\substack{r+s=k\\r,s:{\rm odd}}} q_{r,s}(f)- G_k(q)\sum_{\substack{r+s=k\\r,s:{\rm even}}} q_{r,s}(f) \left(1 + \frac{\beta_r\beta_s}{\beta_k}  \right) \\
& = 6 \sum_{\substack{r+s=k\\r,s:{\rm odd}}} q_{r,s}(f) G_{r,s}^\frac{1}{2} (q),
\end{align*}
where for the last equality we have used \eqref{eq:KZ} (the last equality holds only for a cusp form, but other equalities hold for $f\in M_k$).
Finally, using Theorem 5.1 of \cite{Popa} at $m=0$, we see that the identity
\[\sum_{\substack{r+s=k\\r,s:{\rm even}}} q_{r,s}(f) \left(P_{r,s}(q) - \frac{\beta_r\beta_s}{\beta_k} G_k(q) \right) = \frac{3}{2} \frac{(-1)^{\frac{k}{2}} L_f^\ast(k-1)}{(k-2)!}f\]
holds for any normalized Hecke eigenform $f\in S_k$.
We complete the proof.
\end{proof}

It can be shown that $G_{r,s}^\frac12(q)$'s with $r,s$ odd are linearly independent over $\C$ (the proof is done by using the Vandermonde determinant as in \cite[Theorem 5]{KT}).
Hence Theorem \ref{main} gives an explicit formula for a normalized Hecke eigenform of $S_k$ in terms of the chosen basis of the space $\mathcal{DE}_k$.
Here the notation $G^\frac12$ originates from Yamamoto's study \cite{Yamamoto} of an interpolation of multiple zeta and zeta-star values and our work is motivated by an unpublished work on $\zeta^\frac12$ by Herbert Gangl and Wadim Zudilin.

\

Theorem \ref{main} provides a natural and intrinsic characterization of linear relations among double zeta values obtained from cusp forms.

\begin{corollary}{\rm (\cite{MT})}
For $f\in S_k$, we have
\[ \sum_{\substack{r+s=k\\r,s\ge1:{\rm odd}}} q_{r,s}(f)\zeta^\frac12(r,s)  =0 ,\] 
where we write $\zeta^\frac12(r,s)=\zeta(r,s)+\frac12\zeta(r+s)$.
\end{corollary}

\begin{proof}
We may derive the statement from taking constant terms as $q$-series on both sides of \eqref{eq:main} and the linearity $q_{i,j}(f+g)=q_{i,j}(f)+q_{i,j}(g)$ for $f,g\in S_k$.
\end{proof}

Let us illustrate an example of Theorem \ref{main}.
For $\Delta(q)=q\prod_{n>0}(1-q^n)^{24}=\sum_{n>0} \tau(n)q^n\in S_{12}$, after multiplication by $\frac{1}{680}\frac{4\cdot 8!}{(-1)^6 L_\Delta^\ast (11)}$ we get
\begin{equation}\label{eq:example}
\begin{aligned}
\frac{1}{680} \Delta(q) &= 22680 G_{9,3}^\frac12 (q)-35364 G_{7,5}^\frac12 (q) -29145 G_{5,7}^\frac12 (q) \\
&+13006 G_{3,9}^\frac12 (q) +22680 G_{1,11}^\frac12 (q),
\end{aligned}
\end{equation}
where we have used $L_\Delta^\ast(1)/L_\Delta^\ast(11)=1, \ L_\Delta^\ast(3)/L_\Delta^\ast(11) =L_\Delta^\ast(9)/L_\Delta^\ast(11)  =\frac{691}{1620}$ and $L_\Delta^\ast(5)/L_\Delta^\ast(11) =L_\Delta^\ast(7)/L_\Delta^\ast(11)  =\frac{691}{2520}$.

\

Observe that all coefficients of $G_{r,s}^\frac12(q)$ in \eqref{eq:example} are congruent to 568 modulo 691, which shows
\begin{equation}\label{eq:cong}
\frac{q_{1,11}(\Delta)}{L_\Delta^\ast(11)}\equiv \frac{q_{3,9}(\Delta)}{L_\Delta^\ast(11)}\equiv \frac{q_{5,7}(\Delta)}{L_\Delta^\ast(11)}\equiv \frac{q_{7,5}(\Delta)}{L_\Delta^\ast(11)}\equiv  \frac{q_{9,3}(\Delta)}{L_\Delta^\ast(11)} \not\equiv  0\mod{691}.
\end{equation}
This congruence is motivated by the Ramanujan congruence $\tau(n)\equiv \sigma_{11}(n) \mod{691}$. 
Unfortunately, the Ramanujan congruence is not an obvious consequence of the identity \eqref{eq:example} (we have to do some extra works on $\Q$-linear relations among $g_k(q)$'s, whose coefficient is $\sigma_{k-1}(n)$, and $g_{r,s}(q)$'s), but is obtained as a consequence of Popa's result \cite[Theorem 5.1]{Popa} for $\Delta(q)$, related to Manin's coefficient theorem.
Interestingly, the congruence \eqref{eq:cong} extends to $k\ge12$ with $\dim_\C S_k=1$.
Namely, for $p>k$ such that $p | B_k$, one can show that $q_{r,s}(f)/L_f^\ast(k-1)\equiv q_{r',s'}(f)/L_f^\ast(k-1) \mod{p}$ holds.
This congruence may be related to the Ramanujan-type congruences $a_f(n)\equiv \sigma_{k-1}(n) \mod{p}$ for a normalized Hecke eigenform $f=\sum_{n>0}a_f(n)q^n\in S_k$, where $p$ is a prime such that $p | B_k$ and $p>k$ (see e.g.~\cite[Section 5]{Popa}).
We may also expect a connection between \eqref{eq:cong} and Ihara's congruence \cite[p.258]{Ihara} (taken from \cite[Section 8]{B})
\[ \{\sigma_{3},\sigma_9\}-3\{\sigma_5,\sigma_7\}\equiv 0 \mod 691\]
of generators $\sigma_{2i+1}$ of the motivic Lie algebra $\mathfrak{g}^{\mathfrak{m}}$, since $\mathfrak{g}^{\mathfrak{m}}$ is one of dual structures of the algebra of multiple zeta values.

\

\noindent
{\bf Acknowledgement}.
Many thanks to Masanobu Kaneko for helpful comments on an earlier version of this paper.
This work is partially supported by JSPS KAKENHI Grant No. 18K13393.


\end{document}